\documentclass[11pt]{amsart}

\usepackage{ytableau}

\usepackage{amsmath, amsxtra, amsthm, amssymb,mathtools,bm,amsfonts}
\usepackage{mathrsfs}
\usepackage[normalem]{ulem}
\usepackage[mathscr]{euscript}
\usepackage{graphicx}
\usepackage{float}
\usepackage{url}
\usepackage{color}
\usepackage{bbm}
\usepackage{tikz-cd}
\usepackage{tikz}

\usepackage[T1]{fontenc}

\usepackage{enumitem}
\usepackage{comment}
\usepackage[pdftex,hidelinks]{hyperref}
\hypersetup{
    colorlinks,
    citecolor=magenta,
    filecolor=magenta,
    linkcolor=blue,
    urlcolor=black
}
\usepackage{cleveref}

\usepackage{xcolor}

\renewcommand{\emptyset}{\varnothing}

\theoremstyle{definition}
\newtheorem{thm}{Theorem}[section]
\newtheorem{cor}[thm]{Corollary}
\newtheorem{lem}[thm]{Lemma}

\newtheorem{prop}[thm]{Proposition}
\newtheorem{defn}[thm]{Definition}

\newtheorem{eg}[thm]{Example}
\newtheorem{rem}[thm]{Remark}

\newtheorem{question}[thm]{Question}
\newtheorem{maintheorem}{Theorem}
\newtheorem{fact}[thm]{Fact}

\numberwithin{equation}{section}

\newcommand{\sym}[1]{\operatorname{Sym}_{#1}} 

\newcommand{\supp}{\operatorname{supp}} 

\newcommand{\suchthat}{\;|\;}

\newcommand{\schub}[1]{\mathfrak{S}_{#1}} 

\newcommand{\des}[1]{\operatorname{Des}(#1)} 

\newcommand{\idem}{\operatorname{id}} 

\newcommand{\tope}[1]{\mathsf{T}_{#1}}
\newcommand{\topeik}[2]{\mathsf{T}_{#1,#2}} 
\newcommand{\rope}[1]{\mathsf{R}_{#1}} 
\newcommand{\ropeik}[2]{\mathsf{R}_{#1,#2}} 

\newcommand{\ltfor}[2][] 
{
{\ifx&#1&%
  {\mathsf{LTFor}_{#2}}
\else
  {\mathsf{LTFor}_{#2}^{#1}}
\fi}
}

\newcommand{\rtfor}[2][] 
{
{\ifx&#1&%
  {\mathsf{RTFor}_{>#2}}
\else
  {\mathsf{RTFor}_{>#2}^{#1}}
\fi}
}

\newcommand{\suppfor}[2][] 
{
{\ifx&#1&%
  {\mathsf{Forest}_{#2}}
\else
  {\mathsf{For}_{#2}^{#1}}
\fi}
}

\newcommand{\fl}[1]{\mathrm{Fl}_{#1}}
\newcommand{\GL}{\mathrm{GL}}

\newcommand{\qscoinv}[2][]{
{\ifx&#1&%
  {\operatorname{QSCoinv}_{#2}}
\else
  {{}^{#1}\!\operatorname{QSCoinv}_{#2}}
\fi}
}

\definecolor{ao}{rgb}{0.0, 0.5, 0.0}

\newcommand{\ev}{\operatorname{ev}}

\renewcommand\emph[1]{\textcolor{blue}{\textit{#1}}} 

\newcommand{\DD}{\operatorname{DD}}

\title[Long range divided differences, clusters, and Graham-positivity]{Long range divided differences, clusters, and Graham-positivity}

\author{Hunter Spink}
\address{Department of Mathematics,
University of Toronto, Toronto, ON M5S 2E4, Canada}
\email{\href{mailto:hunter.spink@utoronto.ca}{hunter.spink@utoronto.ca}}

\author{Vasu Tewari}
\address{Department of Mathematical and Computational Sciences, University of Toronto Mississauga, Mississauga, ON L5L 1C6, Canada}
\email{\href{mailto:vasu.tewari@utoronto.ca}{vasu.tewari@utoronto.ca}}

\thanks{
HS and VT acknowledge the support of the NSERC, respectively [RGPIN-2024-04181] and [RGPIN-2024-05433].}

\newcommand{\xl}{{\bm{x}}}
\newcommand{\tl}{{\bm{t}}}

\newcommand{\gl}[1]{\mathrm{GL}_{#1}}

\begin{document}

\begin{abstract}
We study torus-orbit closures in the type $A$ complete flag variety naturally associated to cones in the positive cluster fan, together with their left $S_n$-translates. The torus-equivariant degree maps can be computed via composites of long-range divided difference operations encoded by noncrossing alternating forests, and we give combinatorial algorithms to expand the torus-equivariant homology classes into Graham-positive combinations of Schubert cycles. As applications we obtain combinatorial Graham-positive Schubert cycle expansions for all torus-invariant curves (generalizing the AJS--Billey formula for torus-fixed points), generic torus-orbit closures, and  Richardson varieties for Bruhat intervals $[w,wc']$ where $c'\le s_{n-1}s_{n-2}\cdots s_1$. Projecting to Grassmannians we also obtain Graham-positive Grassmannian Schubert cycle decompositions of torus-orbit closures associated to lattice path matroids on permuted ground sets.
\end{abstract}

\maketitle

\section{Introduction}

Let $B\subset \gl{n}$ denote the upper triangular matrices, and let $(\mathbb{C}^*)^n\cong T\subset B$ be the diagonal matrices. 
In this paper we establish combinatorial Graham-positivity results \cite{Gra01} for new classes of toric subvarieties of the complete flag variety $\fl{n}\coloneqq \gl{n}/B$ by exploiting the clusters of roots of Fomin--Zelevinsky \cite{FZ03,FZ03b} which generate the cones of the type $A$ cluster fan. These include all $T$-invariant curves in $\fl{n}$, generalizing the classical AJS--Billey expansion \cite{AJS,Bil99} for $T$-invariant points, as well as generic $T$-orbit closures. As a corollary we obtain combinatorial Graham-positivity for new families of toric subvarieties of Grassmannians.

Our results crucially hinge on a pair of miraculous identities (Theorem~\ref{thm:Magic}) for type $A$ long-range divided differences inspired by combinatorial Graham-positivity results for the quasisymmetric flag variety \cite{BGNST1,BGNST2}, which we recurse in a tightly controlled manner. 
We leave as an open question how to isolate the manifestly Graham-positive output of our recursion through non-recursive combinatorial formulas.

\subsection{Combinatorial Graham positivity}
Recall that the Bruhat decomposition decomposes the flag variety into Schubert cells $\mathring{X}^w\coloneqq BwB/B\cong \mathbb{A}^{\ell(w)}$ indexed by $w\in S_n$. The closures are the Schubert cycles $X^w\coloneqq \overline{BwB/B}$, and they give a free $H^\bullet_T(pt)\cong \mathbb{Z}[\tl_n]$-basis for the $T$-equivariant Borel--Moore homology group $H_\bullet^T(\fl{n})$. Graham \cite[Theorem 3.2]{Gra01} showed every $T$-invariant subvariety $X\subset \gl{n}/B$ has its fundamental class in $H_\bullet^T(\fl{n})$ expand, for geometric reasons, \emph{Graham-positively} (under the standard Schubert calculus conventions) as
$$[X]= \sum_{w\in S_n}
f^X_w(\tl_n)[X^w]\text{ with }f^X_w(\tl_n)\in \mathbb{N}[t_2-t_1,\ldots,t_n-t_{n-1}].$$
Even for the simplest $T$-invariant varieties in $\fl{n}$ a manifestly nonnegative combinatorial formula $f^X_w(\tl_n)$ is challenging to determine.
To place our results in context, we recall two cases, one where combinatorial Graham-positivity is known and the other where it is not, illustrating this difficulty.

\begin{enumerate}
    \item 
    The $T$-fixed points of $\fl{n}$ are the permutation matrices, indexed by $u \in S_n$. For the corresponding zero-dimensional subvariety $X = \{u\}\subset\fl{n}$, the coefficient $f^X_w(\tl_n)$ equals the double Schubert polynomial evaluation $\schub{w}(\tl_{u};\tl_n)$ where $\tl_u\coloneqq (t_{u(1)},\ldots,t_{u(n)})$. 
    Graham-positivity in this case is the content of the AJS--Billey formula \cite{AJS, Bil99}, whose standard proof requires intricate combinatorics of the nil-Hecke ring (see also~\cite{GK21}).
    
    \item Let $w_{\circ}$ be the longest element, and let $X_u=w_{\circ}X^u$ denote the opposite Schubert cycle. For $u\le v$ in the Bruhat order, the Richardson variety $X=X^v_u=X^v\cap X_u$ has $f^X_w(\tl_n)=a^v_{w,u}(\tl)$, the structure constants of double Schubert polynomial multiplication,
$$\schub{w}(\xl;\tl)\,\schub{u}(\xl;\tl)=\sum a^v_{w,u}(\tl)\,\schub{v}(\xl;\tl),$$
which is famously open to show combinatorial nonnegativity even after setting $\tl\mapsto 0$, giving the structure constants of ordinary Schubert polynomial multiplication. 
Special cases of positivity, both ordinary and Graham, are known; see~\cite{CFMY26, GZ26, Hua23,  KT03, KY04, KZthree, KN01, PW24, Sot96} for a non-exhaustive sample.
To emphasize a contrast with our results, we note that these special cases give a single coefficient $a^v_{w,u}$ for many triples $(u,v,w)$, rather than \textit{all} coefficients of one Richardson variety $X^v_u$. Our results are of the latter kind.
\end{enumerate}

A useful variant of the Graham-positivity question involves the coproduct map on homology, which is dual to the product on cohomology. Concretely, consider the diagonal map $\Delta:\fl{n}\hookrightarrow \fl{n}^2$. Then there is a pushforward map on $T$-equivariant Borel--Moore homology
$$\Delta_*:H_\bullet^T(\fl{n})\to H_\bullet^T(\fl{n}^2)\cong H_\bullet^T(\fl{n})\otimes_{\mathbb{Z}[\tl_n]}H_\bullet^T(\fl{n}).$$
Combining \cite[Theorem 3.2]{Gra01} with the Graham-positivity of the coefficients $a^v_{w,u}(\tl)$ above, for any $T$-invariant subvariety $X\subset \fl{n}$ there is a Graham-positive expansion into the product Schubert cycle basis
$$\Delta_*[X]=\sum f^X_{w,w'}(\tl_n)[X^w]\otimes [X^{w'}]\text{ with }f^X_{w,w'}(\tl_n)\in \mathbb{N}[t_2-t_1,\ldots,t_n-t_{n-1}].$$
Similarly to above, combinatorializing these coefficients is challenging even in the simplest cases. For $X=X^v$, there is an equality $f^X_{w,u}(\tl_n)=a^v_{w,u}(\tl)$ with the generalized Littlewood--Richardson coefficient.
\subsection{Combinatorial Graham positivity associated to clusters}
The \emph{positive cluster fan} is a fan subdividing the positive root cone $\mathbb{R}_{\ge 0}\Phi^+$ whose cones are those collections of positive roots which are contained inside some Fomin--Zelevinsky cluster. 
Identifying a positive root $e_i-e_j$ with the pair $(i,j)$, a cone in this fan is exactly the data of the spreads of a \emph{noncrossing alternating forest (NAF)} $F$ as defined by Gelfand--Graev--Postnikov \cite{GGP97}. See the left panel of Figure~\ref{fig:ncaf_lrdd}. \begin{figure}[!ht]
    \centering
    \includegraphics[scale=0.8]{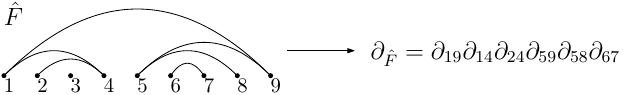}
    \caption{A noncrossing alternating forest $F$ (left) and the associated composite of long-range divided differences $\partial_{F}$ (right).}
    \label{fig:ncaf_lrdd}
\end{figure}

On the other hand, there is an action of $S_n$ on the ring $\mathbb{Z}[\tl_n][\xl_n]$ by permuting the $x$-variables $f(\xl_n;\tl_n)\mapsto f(\xl_w;\tl_n)$ where $\xl_w\coloneqq (x_{w(1)},\ldots,x_{w(n)})$. For a positive root $\alpha=e_i-e_j\in \Phi^+$, the associated \emph{long-range divided difference operation} is defined by $\partial_\alpha=\frac{f-(i,j)\cdot f}{x_i-x_j}$. For $j=i+1$ this is the divided difference operation $\partial_i$ used in the recursive definition of double Schubert polynomials. The primary combinatorial object of study for us is the product $$\partial_{F}\coloneqq \prod_{e_i-e_j\in F}\partial_{ij}$$
taken so that if $[i,j]\subset [k,\ell]$ then $\partial_{ij}$ appears after $\partial_{k\ell}$, see the right panel of Figure~\ref{fig:ncaf_lrdd}. 
 (This product is well defined: arcs of a NAF that are incomparable under containment have disjoint supports, so the corresponding operators commute.)
\begin{maintheorem}
\label{thm:mainA}
\leavevmode
\begin{enumerate}
    \item The expression $(\partial_{F}\schub{w})(\tl_{\sigma};\tl_n)$ is combinatorially Graham-positive.
    \item There is an expansion
    $$\partial_F(fg)(\tl_{\sigma};\tl_n)=\sum_{F',F''}f^{F,F',F''}_{\sigma,\sigma',\sigma''}(\tl_n)(\partial_{F'}f)(\tl_{\sigma'};\tl_n)(\partial_{F''}g)(\tl_{\sigma''};\tl_n)$$
    with each $f^{F,F',F''}_{\sigma,\sigma',\sigma''}(\tl_n)$ combinatorially Graham-positive.
\end{enumerate}
\end{maintheorem}
\begin{cor} $\partial_F(\schub{w_1}\cdots \schub{w_k})(\tl_{\sigma};\tl_n)$ is combinatorially Graham-positive for any $w_1,\ldots,w_k\in S_n$.
\end{cor}
Our second main theorem interprets this result geometrically. Consider the \emph{cluster chart} $U_{F}^-\subset U^-$ within the lower unitriangular matrices as defined in \cite{BGNST2}, the coordinate subgroup of $U^-$ indexed by the negations of the roots in $F$.
For the NAF in \Cref{fig:ncaf_lrdd} we have
$$U_F^-\coloneqq \begin{bmatrix}
1&0&0&0&0&0&0&0&0\\
0&1&0&0&0&0&0&0&0\\
0&0&1&0&0&0&0&0&0\\
\ast&\ast&0&1&0&0&0&0&0\\
0&0&0&0&1&0&0&0&0\\
0&0&0&0&\ast&1&0&0&0\\
0&0&0&0&0&0&1&0&0\\
0&0&0&0&\ast&0&0&1&0\\
\ast&0&0&0&\ast&0&0&0&1\\\end{bmatrix}B\subset \GL_9/B.$$
Because $T$ acts on the coordinates in $U_F^-$ by the negatives of the roots in $F$ and $F$ consists of linearly independent roots, $T$ acts transitively away from the coordinate hyperplanes of $U_F^-$ which makes the closure $Y(F)\subset \fl{n}$ a toric variety.
\begin{maintheorem}
\label{thm:mainB}
    For a NAF $F\subset \Phi^+$ let $Y(F)$ denote the closure of $U_F^-$. \begin{enumerate}
        \item The Schubert cycle expansion of $[\sigma\cdot Y(F)]$ is given by $$[\sigma\cdot Y(F)]=\sum_{w\in S_n}(\partial_F\schub{w})(\tl_{\sigma};\tl_n)[X^w]\in H_\bullet^T(\fl{n}).$$
        Hence the coefficients in this expansion are combinatorially Graham-positive.
        \item Let $\Delta:\fl{n}\hookrightarrow \fl{n}^2$ be the diagonal inclusion. 
        Then  there is an expansion
        $$\Delta_\bullet[\sigma\cdot Y(F)]=\sum_{F',F''}f^{F,F',F''}_{\sigma,\sigma',\sigma''}(\tl_n)[\sigma'\cdot Y(F')]\otimes[\sigma''\cdot Y(F'')]\in H^\bullet_T(\fl{n}^2)$$
        where $f^{F,F',F''}_{\sigma,\sigma',\sigma''}(\tl_n)$ are as in \Cref{thm:mainA}, and are therefore combinatorially Graham-positive.
    \end{enumerate} 
\end{maintheorem}
\begin{cor}
    Under the inclusion $\Delta^{(k)}:\fl{n}\hookrightarrow \fl{n}^k$, the pushforward $\Delta^{(k)}_\bullet[\sigma\cdot Y(F)]$ decomposes as a combinatorially Graham-positive linear combination of the basis of product cycles $[X^{w_1}]\otimes \cdots \otimes [X^{w_k}]$ for $H^\bullet_T(\fl{n}^k)$. 
\end{cor}
The second part of the above geometric result implies that the $\mathbb{Z}[\tl_n]$-submodule
$$H_\bullet^T(\fl{n})_{NAF}\coloneqq \operatorname{Span}_{\mathbb{Z}[\tl_n]}\{[\sigma\cdot Y(F)]\suchthat \sigma\in S_n\text{ and }F \text{ a NAF}\}\subset H_\bullet^T(\fl{n})$$ is in fact a sub-coalgebra of $H_\bullet^T(\fl{n})$, and it  exhibits a surprising Graham-positivity property with respect to this generating set. We believe this coalgebra is worth further study, and leave as an open question the determination of a natural $\mathbb{Z}[\tl_n]$-basis (\Cref{question:coalgebra}).

\subsection{Applications}
We now come to applications of our main results.
\begin{maintheorem}
\label{thm:main1}The following have combinatorially Graham-positive expansions in $\fl{n}$
\begin{enumerate}[label=(\arabic*)]
    \item \label{it1} All $T$-invariant curves. These are the unique curves $w\cdot \mathcal{C}_{ij}$ connecting $w$ to $w(i,j)$.
    \item \label{it3} The Coxeter Richardson varieties $X^{wc'}_w$ for $c'=s_{i_1}s_{i_{2}}\cdots s_{i_k}$  with $i_1>i_2>\cdots>i_k$ and $\ell(wc')=\ell(w)+\ell(c')$, together with their $S_n$-translates $uX^{wc'}_w$.
    \item \label{it4} Generic $T$-orbit closures $\overline{T\cdot M}$.
\end{enumerate}
\end{maintheorem}
We note that applying the reversal involution of the type $A$ root system in (2) also resolves the case when $i_1<i_2<\cdots<i_k$. It turns out that cases~\ref{it1}--\ref{it3} are all directly examples of $\sigma\cdot Y(F)$. By \cite[Theorem 6.4]{HHMP}, the class in case~\ref{it4} is the sum of the classes in case~\ref{it3} taken over all $w\in S_n$ with $\ell(wc)=\ell(w)+\ell(c)$ (equivalently $w(n)=n$), where $c=s_{n-1}s_{n-2}\cdots s_1$.
\begin{rem}
    Even for a single special case---a $T$-invariant curve, say, or a generic $T$-orbit closure---the algorithm passes through the Graham-positivity of \textit{all} translated cluster charts.
\end{rem}
The study of Schubert cycle expansion coefficients is more well-studied in the Grassmannian $\operatorname{Gr}(r;n)=\GL_n/(\GL_r\times \GL_{n-r})$. There are analogously defined Grassmannian Schubert and opposite Schubert cycles $X^\lambda$ and $X_\lambda$ obtained as the closures of $B$ and $B^-$-orbits of $r$-Grassmannian permutations $w_\lambda$ associated to partitions $\lambda\subset r\times (n-r)$. The $T$-equivariant homology class $[\overline{T\cdot x}]\in H_\bullet^T(\operatorname{Gr}(r;n))$ depends only on the discrete combinatorial data of the \emph{matroid} of $x$, i.e. the non-vanishing Pl\"ucker coordinates, and expands into Grassmannian Schubert cycles $X^\lambda$, with one such description given by Berget and Fink \cite{BF22}. The following theorem will follow almost immediately from our results in $\fl{n}$.

\begin{maintheorem}
\label{thm:main2}
The following have combinatorially Graham-positive expansions in $\operatorname{Gr(r;n)}$.
\begin{enumerate}
    \item All $T$-invariant curves.
    \item The $S_n$-translates of the Richardson variety $uX_\mu^\lambda$ for $\lambda/\mu$ a ribbon shape.
    \item Generic $T$-orbit closures in translated Richardson varieties $uX^\lambda_\mu$.
\end{enumerate}
\end{maintheorem}
    The last case subsumes the second case and corresponds to $x$ having its matroid being an $S_n$-translate of a \emph{lattice path matroid}. We in fact prove the last case as a direct corollary of the second case in a way which does not generalize to the complete flag variety beyond the completely generic $T$-orbit closure (where the translation is irrelevant).

\section{Preliminaries}
\label{sec:reminders}
We work with two alphabet sets $\xl=(x_1,x_2,\ldots)$ and $\tl=(t_1,t_2,\ldots)$, the \emph{non-equivariant} and \emph{equivariant} variables respectively, and write $\xl_n,\tl_n$ for their finite truncations $(x_1,\ldots,x_n)$ and $(t_1,\ldots,t_n)$.
\subsection{Torus and characters}
 We identify the characters $\operatorname{Char}(T)\cong \mathbb{Z}^n$ where $\operatorname{Diag}(z_1,\ldots,z_n)\mapsto z_1^{a_1}\cdots z_n^{a_n}$ corresponds to $(a_1,\ldots,a_n)=\sum a_ie_i$.

 We have $H^\bullet_T(pt)\cong \operatorname{Sym}^\bullet \operatorname{Char}(T)\cong \mathbb{Z}[\tl_n]$ where, following the standard convention in algebraic combinatorics, $t_i$ corresponds to inverse character $-e_i=\chi_i^{-1}$. 
 Write $\Phi^+\coloneqq \{e_i-e_j\suchthat i<j\}\subset \mathbb{Z}^n$ for the set of positive roots, $\alpha_i=e_i-e_{i+1}\in \Phi^+$ for $1\le i \le n-1$ the simple roots, and $\Phi=\Phi^+\sqcup (-\Phi^+)$ for all of the roots. For any root $\alpha=e_i-e_j\in \Phi$ and $w\in S_n$, $w$ acts on $\alpha$ by $$\alpha^w\coloneqq w\cdot \alpha=w\cdot (e_i-e_j)=e_{w(i)}-e_{w(j)}.$$
 Under our sign conventions, the simple root $\alpha_i$ corresponds to $t_{i+1} - t_i$,
 and we denote  the set of \emph{Graham-positive polynomials} by
 \[
    \mathbb{N}[\Phi^+]=\mathbb{N}[\alpha_1,\alpha_2,\ldots,\alpha_{n-1}]=\mathbb{N}[t_2-t_1,\ldots,t_n-t_{n-1}]\subset H^\bullet_T(pt).
 \]
\subsection{Double Schubert polynomials}
 The divided difference operators $\partial_i=\frac{\idem-s_i}{x_i-x_{i+1}}\in \operatorname{End}(\mathbb{Z}[\xl_n])$ satisfy the nil-Hecke relations
$$\partial_i^2=0\qquad \partial_i\partial_{i+1}\partial_i=\partial_{i+1}\partial_i\partial_{i+1}\qquad \partial_i\partial_j=\partial_j\partial_i\text{ for }|i-j|\ge 2.$$
Consequently, choosing any reduced word $w=s_{i_1}\cdots s_{i_\ell}$ there is a well-defined composite $\partial_w\coloneqq \partial_{i_1}\cdots \partial_{i_\ell}$, and if $s_{j_1}\cdots s_{j_k}$ is not reduced then $\partial_{j_1}\cdots \partial_{j_k}=0$.

Let $\des{w}=\{i\suchthat w(i)>w(i+1)\}$. The double Schubert polynomials $\{\schub{w}(\xl_n; \tl_n) \suchthat w \in S_n\}$ are characterized by $\schub{w_{\circ}}(\xl_n; \tl_n) = \prod_{i+j \le n}(x_i - t_j)$, where $w_{\circ} \in S_n$ is the longest element, together with the recursion
$$\partial_i \schub{w} = \begin{cases} \schub{ws_i} & i\in \des{w} \\ 0 & \text{otherwise.} \end{cases}$$
This definition is stable: setting $S_\infty = \langle s_1, s_2, \ldots \rangle = \bigcup_{n \ge 0} S_n$, the polynomial $\schub{w}(\xl; \tl)$ is independent of the ambient $n$, so we may index $\schub{w}(\xl; \tl)$ by $w \in S_\infty$.

For $\sigma\in S_n$ we write $$\ev_{\sigma}f(\xl_n;\tl_n)\coloneqq f(\tl_{\sigma};\tl_n)\text{ where }\tl_{\sigma}\coloneqq (t_{\sigma(1)},\ldots,t_{\sigma(n)})$$ The double Schubert polynomials form a basis for the ring of polynomials $\mathbb{Z}[\tl][\xl]$, characterized by the duality 
\begin{equation}
\label{eqn:KroneckerDual}\ev_{\idem}\partial_w\schub{w'}=\delta_{w,w'}.
\end{equation}

\subsection{Equivariant (co)homology and Graham-positivity}
By work of Borel \cite{Bor53}, the $T$-equivariant cohomology ring $H^\bullet_T(\fl{n})$ has a presentation
$$H^\bullet_T(\fl{n})\cong \mathbb{Z}[\tl_n][\xl_n]/\langle f(\xl_n)-f(\tl_n)\suchthat f\in \sym{n}\rangle.$$
with free $\mathbb{Z}[\tl_n]$-basis of double Schubert polynomials $\{\schub{w}(\xl_n;\tl_n)\suchthat w\in S_n\}$, where $x_1,\ldots,x_n$ are the dual Chern roots of the tautological flag of subspaces $\emptyset\subsetneq\mathcal{F}_1\subsetneq \cdots \subsetneq \mathcal{F}_{n-1}\subsetneq \underline{\mathbb{C}}^n$, i.e. $x_i=c_1^T((\mathcal{F}_i/\mathcal{F}_{i-1})^{\vee})$.

To work with Graham-positive expansions combinatorially, we take the Kostant--Kumar approach \cite{KK86}, identifying $T$-equivariant homology with the dual $\mathbb{Z}[\tl_n]$-module
$$
    H_\bullet^T(\fl{n})=\operatorname{Hom}_{\mathbb{Z}[\tl_n]}(H^\bullet_T(\fl{n}),\mathbb{Z}[\tl_n]).
$$

$H_\bullet^T(\fl{n})$ has a $\mathbb{Z}[\tl_n]$-basis of $\{\ev_{\idem}\partial_w\suchthat w\in S_n\}$, which \eqref{eqn:KroneckerDual} implies is Kronecker dual to the double Schubert polynomials.
In this model the fundamental class $[X]$ of a $T$-invariant variety $X$ is represented by the operation
$$\int_X^T:H^\bullet_T(\fl{n})\to \mathbb{Z}[\tl_n]$$
obtained via the $\mathbb{Z}[\tl_n]$-perfect pairing of $H^\bullet_T(\fl{n})$ and $H_\bullet^T(\fl{n})$. Then $\ev_{\idem}\partial_w$ corresponds to $[X^w]$, and a Graham-positive class corresponds to an element
$$\sum f_w[X^w]=\sum f_w\ev_{\idem}\partial_w\in H_\bullet^T(\fl{n})$$
with each $f_w\in \mathbb{N}[t_2-t_1,\ldots,t_n-t_{n-1}]$. To explain how this relates to double Schubert polynomials, note that applying this operation to $\schub{w}(\xl;\tl)$ recovers $f_w(t_2-t_1,\ldots,t_n-t_{n-1})$, so the Graham-positivity of any operation in $H_\bullet^T(\fl{n})$ can be detected by applying it to the double Schubert polynomial basis.

Finally, the left action of $S_n$ preserves the class of $T$-invariant subvarieties of $\fl{n}$, and 
\begin{equation}
\label{eqn:Snhomology}\int_{\sigma\cdot X}^Tf=(\int_X^T f(\xl_n;\tl_{\sigma^{-1}}))|_{\tl_n\mapsto \tl_{\sigma}}\end{equation}

\subsection{The algebra $\DD_n$}
The \emph{divided difference algebra} is the algebra $\DD_n\subset \operatorname{End}(\mathbb{Z}[\xl_n])$ generated by (multiplication by) $\mathbb{Z}[\xl_n]$ and the divided difference operators $\partial_1,\ldots,\partial_{n-1}$. The identity $s_i=\idem+(x_{i+1}-x_i)\partial_i$ gives a canonical embedding $\mathbb{Z}[S_n]\hookrightarrow \DD_n$. In $\DD_n$ we identify $\alpha_i=x_{i+1}-x_i$, so we can write $$\partial_i=\frac{\idem-s_i}{-\alpha_i}\text{ and }s_i=\idem+\alpha_i\partial_i.$$
By the Leibniz identities $\partial_ix_i=x_{i+1}\partial_i+\idem\text{ and }\partial_ix_{i+1}=x_i\partial_i-\idem$
we can straighten any element $\psi\in \DD_n$ as $\psi=\sum f^\psi_w(\xl_n)\partial_w$. Because we can recover $f^\psi_w(\tl_n)=\ev_{\idem} \psi \schub{w}(\xl_n;\tl_n)$, this expansion is unique and we have
$$\DD_n=\bigoplus_{w\in S_n}\mathbb{Z}[\xl_n]\partial_w.$$

The endomorphisms in $\DD_n$ all descend to $\operatorname{End}_{\mathbb{Z}[\tl_n]}(H^\bullet_T(\fl{n}))$, and there is a surjection $\DD_n \twoheadrightarrow H_\bullet^T(\fl{n})$ given by composing
with the evaluation homomorphism $\ev_{\idem}:f(\xl_n;\tl_n)\mapsto f(\tl_n;\tl_n)$, making $H_\bullet^T(\fl{n})$ into a right module for $\DD_n$. Under this map we have
$$\sum_{w\in S_n} f_w(\xl_n)\partial_w\mapsto \sum_{w\in S_n} f_w(\tl_n)\ev_{\idem}\partial_w=\sum_{w\in S_n}f_w(\tl_n)[X^w].$$

For a general positive root $\alpha=e_i-e_j\in \Phi^+$ with reflection $\tau_\alpha=(i,j)$, write $\alpha=\alpha_p^{\,u}$ for some simple $\alpha_p$ and $u\in S_n$, and set
$$\partial_\alpha\coloneqq u\partial_p u^{-1}=\frac{\idem-\tau_\alpha}{-\alpha},$$
which is independent of the chosen representation.
\begin{rem}For us the notation $\partial_\alpha$ when used tacitly assumes that $\alpha$ is a positive root.
\end{rem}

We will often work in the subalgebra $\DD_n^{\Phi}\subset \DD_n$ generated by the $\alpha_i$ and the $\partial_i$; it contains every $\tau_\alpha$ and hence every $u\in S_n$ and long-range divided difference $\partial_\alpha$.
The key relations are the nil-Hecke relations, the relation $\alpha^u=u\alpha u^{-1}$,
\begin{equation}\label{eqn:conjrelabel}
\partial_\beta\,\tau_\gamma=\tau_\gamma\,\partial_{\tau_\gamma\cdot\beta}\qquad(\beta,\ \tau_\gamma\!\cdot\!\beta\in\Phi^+),
\end{equation}
the relation $\tau_\alpha\partial_\alpha=\partial_\alpha$, $\partial_\alpha\tau_\alpha=-\partial_\alpha$,
and the Leibniz rule
$\partial_\alpha \beta=\beta^{\tau_\alpha}\partial_\alpha-\langle \alpha,\beta\rangle$
where $\langle\alpha,\beta\rangle$ is the usual dot product of the vectors. Using the Leibniz rule we see that every operator $\psi\in \DD_n^{\Phi}$ can be expanded as $\psi=\sum f_w^{\psi}(\alpha_1,\ldots,\alpha_{n-1})\partial_w$ with $f_w^{\psi}\in \mathbb{Z}[\Phi]$, and as above this expansion is unique so
$$\DD_n^\Phi=\bigoplus_{w\in S_n}\mathbb{Z}[\Phi]\partial_w.$$
We identify the Graham-positive elements
$$\bigoplus_{w\in S_n}\mathbb{N}[\Phi^+]\partial_w\subset \DD_n^\Phi,$$
which under the map to $H_\bullet^T(\fl{n})$ descend to the Graham-positive linear combinations of Schubert cycles $[X^w]$.

We conclude with the AJS--Billey formula \cite{AJS,Bil99} (see also~\cite{GK21}), which expands a permutation $u$ viewed in $\DD_n$ Graham-positively. 
Via the surjection $\DD_n\twoheadrightarrow H_\bullet^T(\fl{n})$ it computes the class of the $T$-fixed point $\{u\}$.


\begin{thm}[{AJS--Billey \cite{AJS,Bil99}}]
\label{th:ajs-billey}
    Each $u\in S_n\subset \DD_n^\Phi$ has a combinatorially Graham-positive expansion
    $$u=\sum_{w\in S_n}f^{\{u\}}_w(\alpha_1,\ldots,\alpha_{n-1})\,\partial_w\in \DD_n^\Phi,\qquad f^{\{u\}}_w\in\mathbb{N}[\Phi^+].$$
\end{thm}
\begin{proof}
    Pick $i\in\des{u}$. 
    Using $s_i=\idem+\alpha_i\partial_i$ we have
$$u=us_is_i=us_i(\idem+\alpha_i\partial_i)=us_i+\alpha_i^{\,us_i}\,(us_i)\partial_i,$$
where $\alpha_i^{\,us_i}=us_i(\alpha_i)\in\Phi^+$ since $i\in\des u$. As $\ell(us_i)<\ell(u)$, iterating with $us_i$ terminates and yields coefficients in $\mathbb{N}[\Phi^+]$.
\end{proof}

\begin{eg}
    The inclusion $\iota_{\sigma}:\{pt\}\hookrightarrow \fl{n}$ to the $T$-fixed point $\{\sigma\}\subset \fl{n}$ induces the map $\iota^*_\sigma:H^\bullet_T(\fl{n})\to H^\bullet_T(pt)$ given by $f\mapsto \ev_{\sigma}f$. Indeed, $\iota^*_{\sigma} x_i=\iota^*_{\sigma}c_1^T((\mathcal{F}_i/\mathcal{F}_{i-1})^{\vee})=c_1^T((\mathcal{F}_i/\mathcal{F}_{i-1})^{\vee}|_{\sigma})=t_{\sigma(i)}$ under our conventions that $t_j$ corresponds to the inverse of the $j$'th standard character. By the push-pull formula we have $\int_{\{\sigma\}}^Tf=\int_{pt}^T\ev_{\sigma}f=\ev_{\sigma}f$. Therefore $[\{\sigma\}]=\ev_{\sigma}\in H_\bullet^T(\fl{n})$, and
    $$[\{\sigma\}]=\ev_{\sigma}=\sum \schub{w}(\tl_{\sigma};\tl_n)[X^w]=\sum \schub{w}(\tl_{\sigma};\tl_n) \ev_{\idem}\partial_w,$$
    so the Graham-positivity of $[\{\sigma\}]$ records all AJS--Billey expansions $\schub{w}(\tl_{\sigma};\tl)$.
\end{eg}

\section{Noncrossing forests and cluster charts}
\subsection{Noncrossing forests}
Identify positive roots $\alpha=e_{i}-e_{j}\in \Phi^+$ for $1\leq i<j\leq n$ with arcs $(i,j)$ on vertex set $V=\{1,\ldots,n\}$. The \emph{endpoints} of $\alpha$ are $\alpha_L=i$ and $\alpha_R=j$. For a collection of arcs $\mathcal{C}\subset \Phi^+$ we denote the \emph{support} $$\supp \mathcal{C}=\bigcup_{\alpha\in \mathcal{C}} \{\alpha_R,\alpha_L\}.$$
We shall henceforth interchangeably use terminology of positive roots and arcs.

\begin{defn}
    For \emph{distinct} positive roots $\alpha\neq\beta\in \Phi^+$, they either
    \begin{itemize}
        \item \emph{nest} if $\alpha_L\le \beta_L<\beta_R\le \alpha_R$ (written $\beta\prec \alpha$) or $\beta_L\le \alpha_L < \alpha_R \le \beta_R$. (written $\alpha\prec \beta$),
        \item are \emph{separated} if $\alpha_L<\alpha_R<\beta_L<\beta_R$ or $\beta_L<\beta_R<\alpha_L<\alpha_R$,
        \item \emph{bump} if $\alpha_L< \alpha_R=\beta_L<\beta_R$ or $\beta_L< \beta_R=\alpha_L<\alpha_R$, or
        \item \emph{cross} if $\alpha_L<\beta_L<\alpha_R<\beta_R$ or $\beta_L<\alpha_L<\beta_R<\alpha_R$.
    \end{itemize}
    A collection of roots is \emph{noncrossing} if no two arcs in the collection cross. A \emph{noncrossing forest} is a noncrossing subset $F\subseteq \Phi^+$ whose arc diagram is acyclic (or equivalently the roots are linearly independent). In the figure below the first three arc diagrams show nested arcs, the fourth shows two separated arcs, the fifth shows two bumping arcs, and the sixth shows crossing arcs.
\end{defn}
\begin{center}
\begin{tikzpicture}[scale=0.75]
    \draw (-6,0) arc (0:180:1);
    \draw (-7,0) arc (0:180:0.5);
\end{tikzpicture}\,\,
\begin{tikzpicture}[scale=0.75]
    \draw (-3,0) arc (0:180:1);
    \draw (-3,0) arc (0:180:0.5);
\end{tikzpicture}\,\, 
\begin{tikzpicture}[scale=0.75]
    \draw (0,0) arc (0:180:1);
    \draw (-0.5,0) arc (0:180:0.5);
\end{tikzpicture}\qquad
\begin{tikzpicture}[scale=0.75]
    \draw (3,0) arc (0:180:1);
    \draw (5.25,0) arc (0:180:1);
\end{tikzpicture}\qquad
\begin{tikzpicture}[scale=0.75]
    \draw (8,0) arc (0:180:1);
    \draw (10,0) arc (0:180:1);
    \draw (13,0) arc (0:180:1);
    \draw (14,0) arc (0:180:1);
\end{tikzpicture}
\end{center}

Given a noncrossing forest $G$, we say that  $\alpha\in G$ is \emph{minimal} in $G$ if no $\beta\in G\setminus
\{\alpha\}$ satisfies $\beta\prec\alpha$.
We say $\alpha\in G$ is a \emph{leaf edge} if at least one of $\alpha_L,\alpha_R$ is not the endpoint of any other root in $G$, and call such an endpoint of the leaf edge \emph{free}.

\begin{defn}
\label{def:valid}
A total order $(\alpha_1,\ldots,\alpha_m)$ on the roots of a noncrossing forest $F$ is
\emph{valid} if:
\begin{enumerate}[label=\textup{(\Alph*)}]
\item \label{itA} If $\alpha_i\prec\alpha_j$ then $i>j$, and
\item \label{itB} each root  $\alpha_p$ is a leaf edge in $\{\alpha_1,\ldots,\alpha_{p-1},\alpha_p\}$.
\end{enumerate}
For a valid ordering $\omega$ of a noncrossing forest $G$, we write $\partial^{\,\omega}_G$ for the composite operation $\partial_{\alpha_1}\cdots \partial_{\alpha_m}$.
\end{defn}

\begin{figure}[!ht]
    \centering
    \includegraphics[scale=0.9]{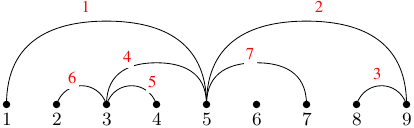}
    \caption{A noncrossing forest $F$ with the arcs endowed with a valid ordering}
    \label{fig:ncaf_valid}
\end{figure}

We now identify a particular class of noncrossing forests that will be especially important.
\begin{defn}A \emph{noncrossing alternating forest} (NAF) is a noncrossing forest $F$ such that no two arcs bump. 
\end{defn}
\begin{rem}It is straightforward to see that any collection $\mathcal{C}\subset \Phi^+$ such that no two arcs bump or cross is automatically a forest, and hence a NAF.
\end{rem}

In a NAF $F$, two arcs sharing an endpoint nest, so arcs that are incomparable under nesting have disjoint supports and their long-range divided differences commute. Thus, a valid ordering on $F$ is simply a linear extension of the nesting order on $F$.
Additionally, all valid orderings of $F$ yield the same operator $\partial_F$, recovering the definition in the introduction. For noncrossing forests $G$ that are not alternating, $\partial_G^{\,\omega}$ may depend on $\omega$.

\begin{thm}
\label{thm:noncrossingequivtoNAF}
    Let $G$ be a noncrossing forest with a valid ordering $\omega=(\gamma_1,\dots,\gamma_m)$.
    Then there exist a permutation $u\in S_n$ and a NAF $F$ with $\supp(F)= \supp(G)$  such that
$\partial_G^{\omega}=u\,\partial_F.$
\end{thm}
\begin{proof}
    We proceed by induction on $m$.
    Let $\gamma\coloneqq \gamma_m$, and write $G'=G\setminus \{\gamma\}$ with valid ordering $\omega'=(\gamma_1,\ldots,\gamma_{m-1})$. Because
 $\omega$ is valid, $\gamma$ is minimal and a leaf edge -- without loss of generality assume that $\gamma_R$ is its free endpoint.  The  inductive hypothesis tells us that $\partial_{G'}^{\omega'}=u'\partial_{F'}$ for some NAF $F'$ with $\supp F' = \supp G'$.
The conditions on $\gamma$ therefore imply that \begin{equation}\label{eqn:G'disjoint}\supp(F')\cap\{\gamma_L+1,\dots,\gamma_R\}=\emptyset.\end{equation}
Now we take two cases.

    Suppose first that $\gamma_L\ne \alpha_R$ for any $\alpha\in F'$.
    Then  $F''=F'\cup \{\gamma\}$ is a NAF with $\gamma$ a minimal arc by \eqref{eqn:G'disjoint}, and $\supp(F'')= \supp(G)$. We conclude $\partial_G^{\omega}=\partial_{G'}^{\omega'}\partial_\gamma=u'\partial_{F'}\partial_{\gamma}=u'\partial_{F''}$ as desired.

    Now suppose that $\gamma_L=\alpha_R$ for some $\alpha\in F'$. By \eqref{eqn:G'disjoint}, the reflection $\tau_\gamma$ acts on the arcs in $F'$ as the order-preserving relabeling taking $\gamma_L\mapsto \gamma_R$ and fixing all other endpoints.
    This carries each root of $F'$ to a positive root and makes $(F')^{\tau_\gamma}$ a NAF, in which $\gamma_L$ is not in the support and $\gamma_R$ is a right endpoint. Therefore $F''\coloneqq (F')^{\tau_\gamma}\cup\{\gamma\}$ is a NAF with $\gamma$ minimal, so by \eqref{eqn:conjrelabel} we have
    \[
        \partial_G^{\omega}=\partial_{G'}^{\omega}\partial_\gamma=u'\partial_{F'}\partial_{\gamma}=u'\partial_{F'}\tau_{\gamma}\partial_{\gamma}=u'\tau_{\gamma}\partial_{(F')^{\tau_{\gamma}}} \partial_{\gamma}=(u'\tau_{\gamma})\,\partial_{F''}
    \]
    as desired.
\end{proof}

\section{A divided difference expansion}


Our goal will be to show through combinatorial means that for any NAF $F$ and $u\in S_n$ we have
\begin{equation}\label{eqn:upartialF}u\partial_F=\sum f^F_w(\alpha_1,\ldots,\alpha_{n-1})\partial_w\in \DD_n^\Phi\end{equation}
with $f^F_w\in \mathbb{N}[\Phi^+]$. This will be the content of \Cref{thm:upartialFpositive}. 
\subsection{A key identity}
The following theorem is our key identity. The perceived asymmetry in the presentation of the two equalities in the theorem, in particular the presence of $+\beta$ in one equation and $-\beta$ in the other equation, is crucial to the proof of \Cref{thm:upartialFpositive}.
\begin{thm}
\label{thm:Magic}
    Suppose $\{\alpha,\beta\}=\{e_i-e_j,e_j-e_k\}$ for some $i<j<k$. Then
    \begin{align*}
\partial_{\alpha+\beta}=&\tau_{\alpha}\partial_{\beta}+\partial_{\alpha}+\beta \partial_{\alpha+\beta}\partial_{\alpha}\\
\partial_{\alpha+\beta}=&\tau_\beta\tau_\alpha\partial_{\beta}+\tau_{\beta}\partial_{\alpha}-\beta \partial_{\alpha+\beta}\partial_{\beta}
\end{align*}
\end{thm}
\begin{proof}
We begin with  the first identity. Note that $\partial_{\alpha+\beta}=\tau_{\alpha}\partial_{\beta}\tau_{\alpha}$ by \eqref{eqn:conjrelabel}. Writing $\tau_{\alpha}=1+\alpha\partial_{\alpha}$,
\begin{align}\label{eq:manipulation_1}
    \partial_{\alpha+\beta}=\tau_{\alpha}\partial_{\beta}(1+\alpha\partial_{\alpha})=\tau_{\alpha}\partial_{\beta}+\tau_{\alpha}\partial_{\beta}\alpha\partial_{\alpha}.
\end{align}
Leibniz's rule implies that $\partial_{\beta}\alpha=(\alpha+\beta)\partial_{\beta}+1$. So substituting this in~\eqref{eq:manipulation_1} gives
\begin{align}\label{eq:manipulation_2}
    \partial_{\alpha+\beta}=\tau_{\alpha}\partial_{\beta}+\tau_{\alpha}((\alpha+\beta)\partial_{\beta}+1)\partial_{\alpha}.
\end{align}
Now note that $\tau_{\alpha}\partial_{\alpha}=\partial_{\alpha}$ and $\tau_{\alpha}(\alpha+\beta)=\beta\tau_{\alpha}$. This transforms~\eqref{eq:manipulation_2} into
\begin{align}
    \partial_{\alpha+\beta}=\tau_{\alpha}\partial_{\beta}+\partial_{\alpha}+\beta\tau_{\alpha}\partial_{\beta}\partial_{\alpha}.
\end{align}
We conclude by noting $\tau_\alpha\partial_\beta\partial_\alpha=\tau_\alpha\partial_\beta\tau_\alpha\partial_\alpha=\partial_{\alpha+\beta}\partial_\alpha$.

As for the second identity, swapping the roles of $\alpha$ and $\beta$ in the first identity yields
\begin{align}\label{eq:symmetric}
    \partial_{\alpha+\beta}=\tau_\beta\partial_\alpha+\partial_\beta+\alpha\,\partial_{\alpha+\beta}\partial_\beta.
\end{align}
To convert to the form in the statement, write $\alpha=(\alpha+\beta)-\beta$ in the last term:
$$\alpha\,\partial_{\alpha+\beta}\partial_\beta=(\alpha+\beta)\partial_{\alpha+\beta}\partial_\beta-\beta\,\partial_{\alpha+\beta}\partial_\beta=(\tau_{\alpha+\beta}-1)\partial_\beta-\beta\,\partial_{\alpha+\beta}\partial_\beta.$$
We have $\tau_{\alpha+\beta}\partial_\beta=\tau_\beta\tau_\alpha\tau_\beta\partial_\beta=\tau_\beta\tau_\alpha\partial_\beta$, and substituting into~\eqref{eq:symmetric} gives the claimed equality.
\end{proof}

\begin{rem}
\label{rem:KacMoody}
The proof of Theorem~\ref{thm:Magic} holds verbatim in the Kostant--Kumar nil-Hecke ring \cite{KK86} of an arbitrary Kac--Moody group, for any pair of roots $\alpha,\beta$ spanning an $A_2$-configuration.
All the type $A$ input in what follows is therefore concentrated in the combinatorics of noncrossing forests of Section~\ref{subsec:noncrossing forests}. 
Finding analogues of this combinatorics adapted to cluster fans of the other types could be interesting.
\end{rem}

\subsection{New forests from old forests}
\label{subsec:noncrossing forests}

To apply our identity recursively, we need the following result.
\begin{prop}
\label{prop:closure}
Let $G$ be a noncrossing forest with valid ordering $(\gamma_1,\ldots,\gamma_m)$, and assume that 
$\gamma\coloneqq \gamma_m$ is non-simple. Suppose we have a decomposition $\gamma=\alpha+\beta$ into positive roots with $\{\alpha,\beta\}=\{(\gamma_L,j),(j,\gamma_R)\}$ for some $\gamma_L<j<\gamma_R$, with the order determined as follows.
\begin{itemize}
\item if $\gamma_R$ is free then $\alpha=(\gamma_L,j)$ and $\beta=(j,\gamma_R)$
\item if $\gamma_L$ is free, then $\beta=(\gamma_L,j)$ and $\alpha=(j,\gamma_R)$.
\end{itemize}
Then $\beta\perp G\setminus\{\gamma\}$, and each of the following is a valid ordering of a noncrossing forest:
\begin{align*}
  C_1=(G\setminus \{\gamma\})\sqcup\{\alpha\}            &:\quad (\gamma_1,\ldots,\gamma_{m-1},\alpha),\\
  C_2=G\sqcup \{\alpha\}    &:\quad (\gamma_1,\ldots,\gamma_{m-1},\gamma_m,\alpha),\\
  C_3=G\sqcup \{\beta\}     &:\quad (\gamma_1,\ldots,\gamma_{m-1},\gamma_m,\beta),\\
  C_4=G^{\tau_\alpha}=(G\setminus \{\gamma\})^{\tau_\alpha}\sqcup\{\beta\}&:\quad
    (\tau_\alpha(\gamma_1),\ldots,\tau_\alpha(\gamma_{m-1}),\tau_\alpha(\gamma_m))=(\tau_\alpha(\gamma_1),\ldots,\tau_\alpha(\gamma_{m-1}),\beta).
\end{align*}
Finally, if $G$ is a NAF, then so are each of $C_1,C_2,C_3,C_4$.

\end{prop}
\begin{proof}
Without loss of generality assume $\gamma_R$ free, the other case is handled symmetrically. Since $\gamma$ is minimal in $G$ with free endpoint $\gamma_R$ we have
\begin{equation}
\label{eqn:GG'support}\supp(G)\cap \{\gamma_L+1,\ldots,\gamma_R-1\}=\emptyset\text{ and }\supp(G\setminus \{\gamma\})\cap \{\gamma_L+1,\ldots,\gamma_R\}=\emptyset.
\end{equation}In particular $j,\gamma_R\not\in \supp(G\setminus \{\gamma\})$, and since $\beta=(j,\gamma_R)$ this verifies that $\beta \perp G\setminus \{\gamma\}$.
    By the first part of \eqref{eqn:GG'support} no arc of $G$ crosses either of $\alpha,\beta$, and the endpoints of $\alpha,\beta$ intersect $\supp G$ in at most one element, which immediately implies $C_1$, $C_2$, and $C_3$ are noncrossing and acyclic. The support statement also guarantees that $\alpha,\beta$ can't nest any arcs in $G$, which verifies that the ordering is indeed valid.

    It remains to analyze $C_4$. By the first part of \eqref{eqn:GG'support}, $\tau_{\alpha}=(\gamma_L,j)$ acts on $G$ as the order-preserving relabeling  fixing all elements in $(\supp G)\setminus \{\gamma_L\}$ and taking $\gamma_L\mapsto j$, and hence $C_4=G^{\tau_{\alpha}}$ is a noncrossing forest whose valid ordering is transported from $G$ via $\tau_\alpha$.

    Finally, suppose that $G$ is a NAF. Because $C_1,\ldots,C_4$ are noncrossing forests, to show they are NAFs it suffices to show that no two arcs bump. $C_4$ was produced by applying an order-preserving map to the endpoints of the arcs in $G$ so the result is clear. For $C_1,C_2,C_3$ it remains to show that the new arcs $\alpha,\beta$ do not bump any existing arc $\delta\in G$. $\alpha,\beta$ are nested under $\gamma$ so cannot bump $\gamma$. For $\delta\ne \gamma$ the support statement \eqref{eqn:GG'support} guarantees $j\not\in \{\delta_L,\delta_R\}$, so bumping with $\alpha$ or $\beta$ could only occur at one of the other endpoints $\gamma_L$ or $\gamma_R$ respectively. But this would imply $\delta,\gamma$ bump, contradicting that $G$ is a NAF.
\end{proof}


\subsection{A Graham-positive recursive algorithm}


We now come to our main combinatorial result straightening $u\partial_F\in \DD_n$ into a Graham-positive combination of $\partial_w$ operators. We note that by \Cref{thm:noncrossingequivtoNAF} it would be no more general to consider $u\partial_G^{\omega}$ for $G$ a noncrossing forest with valid ordering $\omega$.
\begin{thm}
\label{thm:upartialFpositive}
    Let $u\in S_n$ and let $F$ be a NAF. 
    Then $u\partial_F\in \DD_n^\Phi$ expands combinatorially into a Graham-positive linear combination of $\{\partial_w\}_{w\in S_n}$.
\end{thm}
\begin{proof}
 We view $F$ as a noncrossing forest $G$ with valid ordering $\omega=(\gamma_1,\dots,\gamma_m)$, so  $u\partial_F=u\partial_G^{\omega}$.

We run an algorithm maintaining the following invariant: the current expansion is a Graham-positive combination of terms $u'\,\partial_{G'}^{\,\omega'}\,\partial_{w'}$, where $G'$ is a noncrossing forest with valid ordering $\omega'=(\gamma_1',\ldots,\gamma_{m'}')$, $u'\in S_n$, and $w'$ is a (possibly empty) reduced word.
Here $\partial_{G'}^{\,\omega'}$ is the composite still to be reduced and $\partial_{w'}$ is the trailing word already extracted; they play distinct roles. Initially $w'=\mathrm{id}$, so we begin with $u\,\partial_G^{\,\omega}$. The algorithm terminates when every $G'=\emptyset$, at which point each term is $u'\partial_{w'}$.

    Let $\gamma'=\gamma_{m'}'$ be the last arc of $\omega'$, so 
    \[
    \partial_{G'}^{\,\omega'}=\partial_{G'\setminus \{\gamma'\}}^{\,\omega'\setminus \{\gamma'\}}\partial_{\gamma'}.
    \]
    If $\gamma'$ is simple then we absorb $\partial_{\gamma'}$ into the trailing $\partial_{w'}$ and recurse on $u'\partial_{G'\setminus \{\gamma'\}}^{\,\omega'\setminus \{\gamma'\}}$. Otherwise $\gamma'=(i,k)$ is non-simple. Proposition~\ref{prop:closure} applies and yields a decomposition $\gamma'=\alpha'+\beta'$ with $\{\alpha',\beta'\}=\{e_i-e_j,e_j-e_k\}$ together with the four noncrossing forests $$C_1=(G'\setminus \{\gamma'\})\sqcup \{\alpha'\},\quad C_2=G'\sqcup \{\alpha'\},\quad C_3=G'\sqcup \{\beta'\},\quad  C_4=(G')^{\tau_{\alpha'}}=(G'\setminus \{\gamma'\})^{\tau_{\alpha'}}\sqcup \{\beta'\},$$ with each $C_i$ carrying the valid ordering $\omega_i$ produced by Proposition~\ref{prop:closure}. Applying $u'\partial_{G'\setminus \{\gamma'\}}^{\,\omega'\setminus \{\gamma'\}}$ on the left to the two identities of Theorem~\ref{thm:Magic} gives
    \begin{align*}
    u'\partial_{G'}^{\,\omega'}&=u'\tau_{\alpha'} \partial_{C_4}^{\,\omega_4}+u'\partial_{C_1}^{\,\omega_1}+(\beta')^{u'}\,u'\partial_{C_2}^{\,\omega_2},
    \\
    u'\partial_{G'}^{\,\omega'}&=u'\tau_{\beta'}\tau_{\alpha'} \partial_{C_4}^{\,\omega_4}+u'\tau_{\beta'} \partial_{C_1}^{\,\omega_1}+\bigl(-(\beta')^{u'}\bigr)\,u'\partial_{C_3}^{\,\omega_3},
    \end{align*}
    where we used $\partial_{G'\setminus \{\gamma'\}}^{\,\omega'\setminus \{\gamma'\}}\tau_{\alpha'}=\tau_{\alpha'}\partial_{(G'\setminus \{\gamma'\})^{\tau_{\alpha'}}}^{\,\omega_4}$ and $\beta'\perp G'\setminus \{\gamma'\}$ (Proposition~\ref{prop:closure}),  and $u'\beta'=(\beta')^{u'}u'$.
    We apply whichever identity makes the coefficient of the last term ($(\beta')^{u'}$ or $-(\beta')^{u'}$) a positive root.

    We iterate, and claim the recursion terminates, i.e.\ every forest is eventually reduced to $\emptyset$. By the nil-Hecke relations a non-reduced trailing word annihilates its term, so we may discard such terms and assume every surviving $w'$ is reduced, whence $\ell(w')\le\binom{n}{2}$. For a root $\gamma$, its height is defined as $\operatorname{ht}(\gamma)=\gamma_R-\gamma_L$. To each surviving term $u'\partial_{G'}^{\omega'}\partial_{w'}$ assign the pair
\[
   \mu=\Bigl(\tbinom{n}{2}-\ell(w'),\ \operatorname{ht}(\gamma'_{m'})\Bigr),
\]
where $\gamma'_{m'}$ is the last arc of $\omega'$ (and $\operatorname{ht}=0$ if $G'=\emptyset$), ordered lexicographically.

A single step strictly decreases $\mu$. If the last arc is simple, $\partial_{\gamma'}$ is absorbed into the trailing word, so $\ell(w')$ increases and the first coordinate of $\mu$ strictly decreases. 
If the last arc is non-simple, then the trailing word is unchanged.
By Proposition~\ref{prop:closure} the last arc of each of $\omega_1,\ldots,\omega_4$ is $\alpha'$ or $\beta'$, both of height strictly less than $\operatorname{ht}(\gamma')$, so the second coordinate strictly decreases.
In either case $\mu$ drops, so every branch terminates.

The result is a Graham-positive combination of terms $u'\partial_{w'}$. Applying the AJS--Billey expansion (Theorem~\ref{th:ajs-billey}) to each $u'$ separately implies the claim.
\end{proof}
\begin{rem}
    The last part of \Cref{prop:closure} guarantees that the algorithm will always produce NAFs at each intermediate stage.
\end{rem}
\begin{rem}
\label{rem:canonical}
Since $\{\partial_w\}_{w\in S_n}$ is a basis of $\DD_n$ as a $\mathbb{Z}[\xl_n]$-module, the coefficients $f^F_w$ in the expansion  \eqref{eqn:upartialF} of $u\partial_F^{\,\omega}$ are uniquely determined. 
In particular, the Graham-positive expansion produced by the algorithm does not depend on the choices made in the recursion. 
\end{rem}

\begin{cor}[{First part of \Cref{thm:mainA}}]
\label{cor:partialFevalpos}
    For a NAF $F$,  $(\partial_F\schub{w})(\tl_{\sigma};\tl_n)$ is combinatorially Graham-positive.
\end{cor}

\begin{eg}
    Let $F=\{e_1-e_4,e_2-e_4\}$ and $u=s_1s_3$.
    Then the output of the algorithm, before the final AJS--Billey step is applied to each $u'$, is
    \[
        u\partial_F=(x_4-x_1)(x_3-x_1)\partial_{2123}+(x_4-x_1)s_1\partial_{123}+(x_3-x_1)\partial_{213}+(x_3-x_1)\partial_{212}+s_1\partial_{13}+s_1\partial_{12}+\partial_{23}.
    \]
\end{eg}

\section{Geometry}

\subsection{Cluster charts and curves}

As mentioned before, we work in $\fl{n}=\gl{n}/B$ with $B$ the upper-triangular Borel, and let $U^-\subset\gl{n}$ be the group of lower unitriangular matrices.
To a NAF $F$ we attach what we call the  \emph{cluster chart}
$$U_F^-\coloneqq\{g\in U^-\suchthat g_{ji}=0\ \text{ whenever }i<j\text{ and }(i,j)\notin F\}\subseteq U^-.$$
This is the coordinate unipotent subgroup of $U^{-}$ with a free parameter in row $j$, column $i$ for each arc $(i,j)\in F$ and zeros in the remaining subdiagonal positions. 
    We write $Y(F)$ for the closure of $U_F^-$.
For a NAF $F$, the roots are linearly independent, so all fillings with $\ast\in \mathbb{C}^*$ are in the same torus-orbit and so $Y(F)$ is a toric variety.
\begin{thm}
\label{thm:degofsigmaYF}
We have
$$\int_{\sigma Y(F)}^T=\ev_{\sigma} \partial_F$$
\end{thm}
\begin{proof}
   By \eqref{eqn:Snhomology} it suffices to prove the result for $\sigma=\idem$. We need two facts.
    \begin{enumerate}
        \item (\cite{BGG73}) Let $\pi_i:\fl{n}\to \GL_n/P_i$ be the projection, where $P_i$ is the $i$'th minimal parabolic subgroup. Then for $Z\subset \fl{n}$ we have $\int_{\pi^*\pi_*Z}^T=\int_Z^T\partial_i$.
        \item (\cite[Proposition 12.8]{BGNST2}) Let $\Psi_{i}^{-}$ be the inclusion of $\fl{n-1}\hookrightarrow \fl{n}$ obtained by the Bergeron--Sottile pattern map inclusion \cite{BS98}
        $$\begin{bmatrix}A_{(i-1)\times (i-1)}&B_{(i-1)\times (n-i)}\\C_{(n-i)\times (i-1)}&D_{(n-i)\times (n-i)}\end{bmatrix}\mapsto \begin{bmatrix}A_{(i-1)\times (i-1)}&0_{(i-1)\times 1}&B_{(i-1)\times (n-i)}\\
        0_{1\times (i-1)}&1&0_{1\times (n-i)}\\
        C_{(n-i)\times (i-1)}&0_{(n-i)\times 1}&D_{(n-i)\times (n-i)}\end{bmatrix}.$$
        Then for $\rope{i}^-f\coloneqq f(x_1,\ldots,x_{i-1},t_i,x_{i},\ldots,x_{n-1};\tl_n)$ we have $\int_{\Psi_i^-Z}^Tf=\int_Z^{T'}\rope{i}^-f$, where $T'$ is the diagonal torus in $\fl{n-1}$ with standard characters  $t_1,\ldots,t_{i-1},t_{i+1},\ldots,t_n$, and $\int_Z^{T'}$ is extended $\mathbb{Z}[\tl_n]$-linearly.
    \end{enumerate}
    Now, given a NAF $F$, if there is an element $i$ not used in any arc, then there is an associated NAF $F'$ on $n-1$ elements obtained by deleting $i$ and $\Psi_i Y(F')=Y(F)$. Then
    $$\int_{Y(F)}^Tf=\int_{Y(F')}^{T'}\rope{i}^-f=\ev_{\idem}^{T'}\partial_{F'}\rope{i}^-f=\ev_{\idem}^T\partial_F f.$$
    On the other hand if every element of $\{1,\ldots,n\}$ is an arc endpoint, then there must be some $\gamma=e_{i}-e_{i+1}\in F$. In this case, we note that $U_{F}^-$ can be obtained from $U_{F\setminus \gamma}^-$ by adding arbitrary multiples of the $(i+1)$'st column $e_{i+1}$ to the previous column. Because the $(i+1,i+1)$ entry is the only nonzero entry in row and column $i+1$ in $U_{F\setminus\gamma}^-$, this shows that $Y(F)=(\pi_i^*)(\pi_i)_*Y(F\setminus \gamma)$, and so
    \begin{equation*}
        \int_{Y(F)}^Tf=\int_{Y(F\setminus \gamma)}^T\partial_if=\ev_{\idem}\partial_{F\setminus \gamma}\partial_if=\ev_{\idem}\partial_F f.\qedhere
    \end{equation*}
\end{proof}
\begin{cor}[{First part of \Cref{thm:mainB}}]
    We have a combinatorially Graham-positive expansion
    $$[\sigma\cdot Y(F)]=\sum a^u_{F,\sigma}(\tl_n)[X^u]\text{ with }a^u_{F,\sigma}(\tl_n)=(\partial_F\schub{u})(\tl_{\sigma};\tl_n)\in \mathbb{N}[t_2-t_1,\ldots,t_n-t_{n-1}].$$
\end{cor}

We now compute the combinatorial Graham-positive expansion for $T$-invariant curves. Denote $\mathcal{C}_{ij}$ the unique $T$-invariant curve connecting $\idem$ to $(i,j)$. Then every $T$-invariant curve in $\fl{n}$ is of the form $w\cdot \mathcal{C}_{ij}$ for some $w,i,j$.

\begin{thm}
    We have $\int_{u\cdot \mathcal{C}_{ij}}^Tf(\xl_n;\tl_n)=\ev_u \partial_{ij}f(\xl_n;\tl_n)$.
\end{thm}
\begin{proof}
This follows immediately from \Cref{thm:degofsigmaYF} and the observation that $\mathcal{C}_{ij}=Y(\{e_i-e_j\})$.
\end{proof}
\begin{cor}
    For $\alpha=e_i-e_j$ the $T$-invariant curve $w\cdot \mathcal{C}_{ij}$ connecting $w$ to $w(i,j)$ has a combinatorially Graham-positive expansion
    $$[w\cdot \mathcal{C}_{ij}]=\sum a^{w\cdot \mathcal{C}_{ij}}_u(\tl_n)[X^u]\text{ with }a^{w\cdot \mathcal{C}_{ij}}_u(\tl_n)=(\partial_{ij}\schub{u})(\tl_w;\tl_n)\in \mathbb{N}[t_2-t_1,\ldots,t_n-t_{n-1}].$$
\end{cor}
\subsection{Standard Coxeter Richardsons and generic $T$-orbit closures}
Let $c'\le s_{n-1}s_{n-2}\cdots s_1$.
By \cite[Theorem 6.1]{nst_c}, to each Richardson variety $X^{wc'}_w$ with $\ell(w)+\ell(c')=\ell(wc')$ there are a NAF $F$ and a permutation $u_F$, both determined by $c'$, such that $(wu_F)^{-1}X^{wc'}_w=Y(F)$.
\begin{cor}
We have a combinatorially Graham-positive expansion 
    $$[X^{wc'}_w]=\sum a^{wc'}_{w,u}(\tl_n)[X^u]\text{ with }a^{wc'}_{w,u}(\tl_n)=[\schub{wc'}](\schub{u}\schub{w})\in \mathbb{N}[t_2-t_1,\ldots,t_n-t_{n-1}]$$
    a double generalized Littlewood--Richardson coefficient.
\end{cor}
Recall that a generic $T$-orbit closure has moment polytope the permutahedron $$\operatorname{Perm}_n\coloneqq \operatorname{conv}\{w\cdot (\lambda_1,\ldots,\lambda_n)\suchthat w\in S_n\}$$
for $\lambda_1> \cdots > \lambda_n$ a dominant regular weight. By the Atiyah--Bott localization formula we have
$$\int_{X_{Perm}}f(\xl_n;\tl_n)=\sum_{w\in S_n}\frac{f(\tl_w;\tl)}{\prod_{i=1}^{n-1}(t_{w(i)}-t_{w(i+1)})}$$

This last expression is known as the \emph{divided symmetrization} of $f$, and is studied by Postnikov. 
As shown by \cite{HHMP,Li24}, there is a $T$-invariant degeneration $$X_{Perm}\rightsquigarrow \bigcup_{w\in S_n\text{ and }w(n)=n}X^{wc}_w$$
where $c=s_{n-1}s_{n-2}\cdots s_1=(n,n-1,\ldots,1)$ and each component is a toric variety appearing with multiplicity $1$.
\begin{cor}
    $[X_{Perm}]$ has a combinatorially Graham-positive expansion. In particular divided symmetrization of a double Schubert polynomial is Graham-positive.
\end{cor}
\begin{proof}
    $[X_{Perm}]=\sum_{w\in S_n\text{ with }w(n)=n}[X^{wc}_w]$.
\end{proof}

\begin{rem}
More combinatorially, it was shown in \cite[Section 9]{nst_c} that divided symmetrization is
$$f(\xl_n;\tl_n)\mapsto \ropeik{1}{1}\topeik{1}{2}(\topeik{1}{3}+\topeik{2}{3})\cdots (\topeik{1}{n}+\topeik{2}{n}+\cdots + \topeik{n-1}{n})f(\xl_n;\tl_n)$$
in the notation of Appendix~\ref{sec:decorated}. Expanding the product, each summand is a valid composite in the sense of Appendix~\ref{sec:decorated}, which one rewrites into the canonical form $w\partial_F$ using the commutation relations of the augmented Thompson monoid together with the three-term relations of Corollary~\ref{th:two_three_term_relations}; this exhibits each as $w\partial_F$ for various $w,F$.
\end{rem}

\subsection{Torus-orbit closures in the Grassmannian}
Consider the Richardson variety $X_\mu^\lambda\subset \operatorname{Gr}(r;n)$ where the skew shape $\lambda/\mu$ is a ribbon. We recall from \cite[Section 10]{nst_c} how this is a toric variety we can realize as $\pi_*(X^{wc'}_w)=X_\mu^\lambda$ for $w,c'$ (see also \cite{QGr}). Let $b_\mu$, $b_\lambda$ be the binary sequences encoding $\mu$ and $\lambda$ (the sequence $1^r0^{n-r}$ representing the identity). Then because $\lambda/\mu$ is a ribbon we can find $c'$ such that $c'b_\lambda=b_\mu$, where $c'\in S_n$ is a product of backwards cycles on disjoint intervals and $|\lambda/\mu|=\ell(c')$. Then we let $w$ be the permutation in one-line notation obtained by replacing the $1$'s in $b_\mu$ with the numbers $1,2,\ldots,r$ in some order, the $0$'s in $b_\mu$ that are not the $k$ right endpoints of cycles with $r+1,\ldots,n-k$ in some order, and then remaining $0$'s in $b_\mu$ that are the right endpoints of cycles with $n-k+1,\ldots,n$. Since $\pi(X^{wc'}_w)\subseteq X^\lambda_\mu$ and the two varieties have the same dimension, $\pi$ restricts to a generically finite surjection onto $X^\lambda_\mu$; as it is birational, $\pi_*[X^{wc'}_w]=[X^\lambda_\mu]$.

\begin{cor}
    $uX^\lambda_\mu$ has a combinatorially Graham-positive expansion into Grassmannian Schubert cycles
    $$[uX^\lambda_\mu]=\sum_\nu f_\nu(\tl_n)[X^\nu]\in H_\bullet^T(\operatorname{Gr}(r;n))\text{ with }f_\nu(\tl_n)\in \mathbb{N}[\Phi^+]$$
\end{cor}
\begin{proof}
If $w\in S_n$ is an $r$-Grassmannian permutation, i.e. $\des{w}\subset \{r\}$, let $\lambda_w$ be the partition which has $b_{\lambda_w}=w\cdot (1^r0^{n-r})$. The projection map $\pi:\fl{n}\to \operatorname{Gr}(r;n)$ is $\mathbb{Z}[\tl_n]$-linear on $T$-equivariant Borel-Moore homology, and
$$\pi_*[X^w]= \begin{cases}[X^{\lambda_w}]&\des{w}\subset \{r\}\\0&\text{otherwise},\end{cases}$$
so the result follows from the Graham-positive expansion for $[uX^{wc'}_w]\in H_\bullet^T(\fl{n})$.
\end{proof}
\begin{cor}
    If $\Lambda\in \operatorname{Gr}(r;n)$ is a linear subspace whose associated matroid is a lattice path matroid, then there is a combinatorially Graham-positive expansion
    $$[u\cdot \overline{T\cdot \Lambda}]=\sum a_\lambda(\tl_n)[X^\lambda]\in H_\bullet^T(\operatorname{Gr}(r;n))\text{ with }a_\lambda(\tl_n)\in \mathbb{N}[\Phi^+]$$
\end{cor}
\begin{proof}
    We have $u\cdot \overline{T\cdot \Lambda}=\overline{T\cdot (u\Lambda)}$, where $u\Lambda$ has as its matroid an $S_n$-translate of a lattice path matroid. The class  $$[\overline{T\cdot \Lambda}]=\sum a^\Lambda_\lambda(\tl_n)[X^\lambda]\in H_\bullet^T(\operatorname{Gr}(r;n))$$
is known to only depend on the underlying matroid $M$ of $\Lambda$, and to be an additive function of matroid polytopes (there is a more refined statement holding in $T$-equivariant $K$-theory concerning exact inclusion/exclusion with respect to lower dimension polytopes, see \cite[Proof of Proposition 8.3]{BF22}), which in particular implies if the matroid polytope of $M$ is subdivided into the matroid polytopes of $M_1,\ldots,M_k$, then $\sum a^{M_i}_\lambda(\tl_n)=a^M_\lambda(\tl_n)$. By \cite[Lemma 4.3.5]{bidkhori_2010} there is a subdivision of any lattice path matroid polytope into snake matroid polytopes, and translating by $S_n$ subdivides every translated lattice path matroid polytope into translated snake matroid polytopes. Therefore the coefficients for the translated lattice path matroid are sums of the coefficients for the corresponding translated snake matroids as desired.
\end{proof}

\begin{rem}
    As a Schubert matroid is a special case of a lattice path matroid when $\mu=\emptyset$, by \cite{DF10} we now know the combinatorial Graham-positivity for torus-orbit closures associated to a collection of matroids that has the property that the class of any other $T$-orbit closure is obtainable via inclusion--exclusion. Unfortunately this inclusion--exclusion does not preserve Graham-positivity.
\end{rem}
\section{Coproduct}
\label{sec:coproduct}
In this section we work with noncrossing forests with valid orderings -- this is equivalent to working with NAFs by \Cref{thm:noncrossingequivtoNAF}. 
\subsection{The coproduct and the long-range Leibniz rule}

Viewing the divided difference algebra $\DD_n$ as a left $\mathbb{Z}[\xl_n]$-algebra, it carries a coassociative coproduct 
$$\Delta\colon\DD_n\longrightarrow \DD_n\otimes_{\mathbb{Z}[\xl_n]} \DD_n,$$
obtained by Kostant--Kumar \cite{KK86} as the restriction to $\DD_n$ of the natural coproduct on the twisted group algebra $\mathbb{Q}(\xl_n)\rtimes S_n$, under which every $w\in S_n$ is group-like. 
For $g\in \mathbb{Z}[\xl_n]$ and $w\in S_n$ the coproduct is determined by the following identities (in fact just the last one and $\mathbb{Z}[\xl_n]$-linearity suffice):
$$\Delta(g)=g(\idem\otimes\idem)=g\otimes \idem=\idem\otimes g\qquad \Delta(w)=w\otimes w\qquad \Delta(\partial_i)=\partial_i\otimes \idem+s_i\otimes \partial_i=\partial_i\otimes s_i+\idem\otimes \partial_i.$$
It encodes the action of $\DD_n$ on products: if $\Delta(\Phi)=\sum\Phi_{(1)}\otimes\Phi_{(2)}$ then $\Phi(fg)=\sum\Phi_{(1)}(f)\,\Phi_{(2)}(g)$.
One sees that the value on $\partial_i$ is the Leibniz rule $\partial_i(fg)=\partial_i(f)\,g+s_i(f)\,\partial_i(g)$. 
Via the surjection $\DD_n\twoheadrightarrow H_\bullet^T(\fl{n})$ this descends to the coproduct on $H_\bullet^T(\fl{n})$ induced by pushforward along the diagonal.
One key observation that we shall employ is that the Leibniz rule continues to have the same form even for long-range divided differences:
\begin{equation}
\label{eqn:longrangeleibniz}
\Delta(\partial_\gamma)=\partial_\gamma\otimes \idem+\tau_\gamma\otimes \partial_\gamma=\partial_\gamma\otimes \tau_\gamma+\idem\otimes \partial_\gamma\qquad\text{for every }\gamma\in\Phi^+.
\end{equation}

\subsection{Noncrossing forests and the coproduct formula}
Our goal is to compute the coproduct $\Delta(\partial_F)$ where $F$ is a NAF.
The computation is done in two steps: expand it into a sum of pure tensors, then rewrite each tensor factor in the canonical form $u\,\partial_{F'}$, for some permutation $u$ and NAF $F'$.

Let $(\gamma_1,\ldots,\gamma_m)$ be a valid ordering of $F$. Since $\Delta$ is a $\mathbb{Z}[\xl_n]$-algebra homomorphism, the Leibniz rule \eqref{eqn:longrangeleibniz} implies:
\begin{equation}
\label{eqn:LSRS}\Delta(\partial_F)=\prod_{t=1}^{m}\big(\partial_{\gamma_t}\otimes\idem+\tau_{\gamma_t}\otimes\partial_{\gamma_t}\big)=\sum_{S\subset \{1,\ldots,m\}}L_S\otimes R_S\end{equation}
where for $S=\{p_1<\cdots<p_k\}$ we have
$$L_S=\partial_{\gamma_1}\cdots\partial_{\gamma_{p_1-1}}\,\tau_{\gamma_{p_1}}\,\partial_{\gamma_{p_1+1}}\cdots\tau_{\gamma_{p_k}}\,\partial_{\gamma_{p_k+1}}\cdots\partial_{\gamma_m},\qquad R_S=\partial_{\gamma_{p_1}}\cdots\partial_{\gamma_{p_k}}.$$
The second factor is already in final form: $R_S=\partial_{F|_S}$, where $F|_S\coloneqq\{\gamma_s:s\in S\}$ is the sub-forest of selected arcs with the induced ordering $\omega|_S\coloneqq(\gamma_{p_1},\ldots,\gamma_{p_k})$. A sub-forest of a NAF is a NAF, so $R_S$ requires no further work.
\begin{lem}
    If $G_1\sqcup \{\gamma\} \sqcup G_2$ is a noncrossing forest with valid ordering $\omega=(\omega_{G_1},\gamma,\omega_{G_2})$, then $G_1^{\tau_\gamma}G_2$ is a noncrossing forest with valid ordering $\omega_{G_1}^{\tau_\gamma},\omega_{G_2}$.
\end{lem}
\begin{proof}
The fact that $G_1^{\tau_\gamma}$ is noncrossing with valid ordering $\omega_{G_1}^{\tau_\gamma}$ follows from \Cref{prop:closure} applied to $G_1\sqcup \{\gamma\}$ with valid ordering $(\omega_{G_1},\gamma)$. To establish $G_1^{\tau_\gamma}\sqcup G_2$ is a noncrossing forest we have to rule out that a pair of arcs in $\alpha^{\tau_\gamma}\in G_1^{\tau_\gamma}$ and $\beta\in G_2$ are identically equal or cross. Suppose without loss of generality that $\beta_R$ is the free endpoint of $\beta$ in the prefix of $G_1\sqcup \{\gamma\}\sqcup G_2$ that it lies in. Then because $\supp(\alpha^{\tau_\gamma})\subset \supp\{\alpha,\gamma\}$, we have \begin{equation}
\label{eqn:suppalphataugamma}\supp \{\alpha^{\tau_\gamma}\}\cap \{\beta_L+1,\ldots,\beta_R\}=\emptyset\end{equation}
which rules out $\alpha^{\tau_\gamma}=\beta$ and prevents $\alpha^{\tau_\gamma}$ from crossing $\beta$.

We now show that $\omega_{G_1}^{\tau_\gamma}\omega_{G_2}$ is a valid ordering. Because $\omega_{G_1}^{\tau_\gamma}$ and $\omega_{G_2}$ are valid orderings, it suffices to show that $\alpha^{\tau_\gamma}$ cannot be nested under $\beta$, and that $\beta$ is a leaf edge in the prefix of $G_1^{\tau_\gamma}G_2$ that ends in it with $\beta_R$ as its free endpoint. But these both follow from \eqref{eqn:suppalphataugamma}.
\end{proof}
\begin{cor}
\label{cor:LSfix}
    $L_S=u\partial_{F'}$ for some permutation $u\in S_n$ and NAF $F'$.
\end{cor}
\begin{proof}
Applying the lemma iteratively, if $G_1\sqcup \{\gamma_{p_1}\}\sqcup G_2\sqcup \{\gamma_{p_2}\}\sqcup \cdots \sqcup \{\gamma_{p_k}\}\sqcup G_{k+1}$ has valid ordering $\omega_{G_1}\sqcup \{\gamma_1\}\sqcup \cdots \sqcup \{\gamma_k\}\sqcup \omega_{G_{k+1}}$, then $G'=G_1^{s_{\gamma_{p_1}}\cdots s_{\gamma_{p_k}}}G_2^{s_{\gamma_{p_2}}\cdots \gamma_{p_k}}\cdots G_{k}^{\gamma_{p_k}}G_{k+1}$ is a noncrossing forest with valid ordering $\omega'=\omega_{G_1}^{\tau_{\gamma_{p_1}}\cdots \tau_{\gamma_{p_k}}}\omega_{G_2}^{\tau_{\gamma_{p_2}}\cdots \tau_{\gamma_k}}\cdots \omega_{G_k}^{\tau_{\gamma_{p_k}}}\omega_{G_{k+1}}$.
    By writing $$L_S=\partial_{G_1}^{\omega_{G_1}}\tau_{\gamma_{p_1}}\cdots \partial_{G_k}^{\omega_{G_k}}\tau_{\gamma_{p_k}}\partial_{G_{k+1}}^{\omega_{G_{k+1}}}$$
we can use the identity \eqref{eqn:conjrelabel} iteratively to pull out each $\tau_{p_i}$ from $L_S$ to the left from the $i=1$ to $k$ and obtain $L_S=\tau_{p_1}\cdots\tau_{p_k}\partial_{G'}^{\omega'}$ as desired. We then can express this as $u\partial_{F'}$ for some permutation $u$ and NAF $F$ by \Cref{thm:noncrossingequivtoNAF}.
\end{proof}
\begin{thm}[{Second parts of \Cref{thm:mainA} and \Cref{thm:mainB}}]
\label{thm:coproduct}
If $F$ is a NAF and $\sigma\in S_n$ there is an expansion
$$\Delta(\sigma\partial_F)=\sum_{S\subset F}(\sigma u_{F,S}^{(1)}\partial_{G_{F,S}^{(1)}})\otimes (\sigma u_{F,S}^{(2)}\partial_{G_{F,S}^{(2)}})$$
where $u_{F,S}^{(i)}$ are permutations depending on $S$ and $F$, and $G^{(i)}_{F,S}$ are NAFS depending on $F$ and $S$.
\end{thm}
\begin{proof}
    It suffices to prove this for $\sigma=1$ since $\Delta(\sigma)=\sigma\otimes \sigma$. We start with \eqref{eqn:LSRS}. The tensor factor $R_S=\partial_{F|_S}$ is already in the desired form, so we take $u_{F,S}^{(2)}=\idem$ and $G_{F,S}^{(2)}=F|_S$. We then rewrite $L_S=u_{F,S}^{(1)}\partial_{G_{F,S}^{(1)}}$ by applying \Cref{cor:LSfix}, setting $u_{F,S}^{(1)}=u$ and $G_{F,S}^{(1)}=F'$.
\end{proof}
\begin{cor}
\label{cor:productpositive}
    For any NAF $F$, the evaluation $\partial_F(\schub{w}\schub{w'})(\tl_{\sigma};\tl_n)$ is Graham-positive.
\end{cor}
\begin{proof}
    Applying Theorem~\ref{thm:coproduct} shows
    $$\partial_F(\schub{w}\schub{w'})(\tl_{\sigma};\tl_n)=\sum_{S\subset F}\big(\partial_{G_{F,S}^{(1)}}\schub{w}\big)(\tl_{\sigma u_{F,S}^{(1)}};\tl_n)\cdot \big(\partial_{G_{F,S}^{(2)}}\schub{w'}\big)(\tl_{\sigma u_{F,S}^{(2)}};\tl_n).$$
    Each factor is Graham-positive by Corollary~\ref{cor:partialFevalpos}.
\end{proof}
The classes $\ev_\sigma\partial_F$ are far from linearly independent, so distinct summands in Theorem~\ref{thm:coproduct} may coincide or rewrite further. The fact that the coefficients are all $1$ in the expansion of the coproduct suggests there may be a multiplicity-free geometric degeneration that implies this result -- we leave this as an open question. The result shows that there is a sub-coalgebra
$$\operatorname{Span}_{\mathbb{Z}[\tl_n]}\{\ev_{\sigma}\partial_F \suchthat \sigma\in S_n\text{ and }F\text{ a NAF}\}\subset H_\bullet^T(\fl{n})$$
is an $S_n$-invariant sub-coalgebra with nonnegative integer structure constants relative to this spanning family. Taking $F=\emptyset$ shows it contains every fixed-point class $[\{\sigma\}]=\ev_\sigma$, and it contains the classes of all varieties considered in the geometric section, as well as their $S_n$-translates.
\begin{question}
\label{question:coalgebra}
    Is there a $\mathbb{Z}[\tl_n]$ basis for this coalgebra, and can we describe this span as a representation of $S_n$?
\end{question}

\appendix

\section{Diagrammatics}
\label{sec:decorated}
Let $\ropeik{i}{k}f=f(x_1,\ldots,x_{i-1},t_k,x_i,\ldots,x_{n-1};\tl)$, and define the operator $$\topeik{i}{k}f=\ropeik{i}{k}(\partial_{i}f)=\ropeik{i+1}{k}(\partial_i f),$$
where the second equality holds as $\partial_if$ is symmetric in $\{x_i,x_{i+1}\}$.
Say a composite $X_1\cdots X_n$ of $n$ of these operations is \emph{valid} if
\begin{enumerate}
    \item Each of the numbers $1,\ldots,n$ appears as the second subscript of an $X_i$ exactly once.
    \item $X_m$ is of the form $\ropeik{i}{k}$ for some $1\le i \le m$ or $\topeik{i}{k}$ for some $1\le i \le m-1$. In particular $X_1$ must be of the form $\ropeik{1}{k}$. 
\end{enumerate}
These operations satisfy the commutation relations
\begin{align*}
    \topeik{i}{k_1}\topeik{j}{k_2}=\topeik{j}{k_2}\topeik{i+1}{k_1}\text{ for }i>j\qquad \topeik{i}{k_1}\ropeik{j}{k_2}=\ropeik{j}{k_2}\topeik{i+1}{k_1}\text{ for }i\ge j\\
    \ropeik{i}{k_1}\topeik{j}{k_2}=\topeik{j}{k_2}\ropeik{i+1}{k_1}\text{ for }i>j\qquad \ropeik{i}{k_1}\ropeik{j}{k_2}=\ropeik{j}{k_2}\ropeik{i+1}{k_1}\text{ for }i\ge j.
\end{align*}
 These are the defining relations of the augmented Thompson monoid from \cite[Definition 3.7]{nst_c}, except the letters have been decorated with second subscripts.
We represent the operator $f\mapsto (\partial_F f)(\tl_{\sigma};\tl)$ via a decorated nested forest as follows.
\begin{thm}
\label{thm:validcomposite}
    Let $F$ be a noncrossing alternating forest on $[n]$, and let $\sigma\in S_n$.
    $f(\xl_n;\tl_n)\mapsto (\partial_F f)(\tl_{\sigma};\tl_n)$ can be expressed as a valid composite of these operations.
\end{thm}
\begin{proof}
    First, suppose that there is some $i\in \{1,\ldots,n\}$ that is not an arc endpoint. Then if $F'$ is the NAF obtained by mapping the endpoints in $\{1,\ldots,n\}\setminus \{i\}$ to $\{1,\ldots,n-1\}$ via the order-increasing bijection, then 
    $$(\partial_{F}f)(\tl_{\sigma};\tl_n)=(\partial_{F'}(\rope{i,\sigma(i)}f))(t_{\sigma(1)},\ldots,\widehat{t_{\sigma(i)}},\ldots,t_{\sigma(n)};t_1,\ldots,\widehat{t_{\sigma(i)}},\ldots,t_n)$$
    and we conclude by applying the inductive hypothesis since $(\sigma(1),\ldots,\widehat{\sigma(i)},\ldots,\sigma(n))$ is a permutation of $(1,\ldots,\widehat{{\sigma(i)}},\ldots,n)$

    Otherwise, the last arc of $F$ is simple $\gamma=(i,i+1)$ for some $i\in \{1,\ldots,n-1\}$. Then 
    $$(\partial_Ff)(\tl_{\sigma};\tl_n)=\begin{cases}(\partial_{F'}(\tope{i,\sigma(i)}f)(t_{\sigma(1)},\ldots,\widehat{t_{\sigma(i)}},\ldots,t_{\sigma(n)};t_1,\ldots,\widehat{t_{\sigma(i)}},\ldots,t_n)&\gamma_L\text{ free}\\
    (\partial_{F''}(\tope{i,\sigma(i+1)}f)(t_{\sigma(1)},\ldots,\widehat{t_{\sigma(i+1)}},\ldots,t_{\sigma(n)};t_1,\ldots,\widehat{t_{\sigma(i+1)}},\ldots,t_n)&\gamma_R\text{ free,}\end{cases}$$
    where $F'$ is obtained from $F\setminus \gamma$ by mapping the endpoints $\{1,\ldots,n\}\setminus \{\gamma_L\}\to \{1,\ldots,n-1\}$, and $F''$ is obtained from $F\setminus \gamma$ by mapping the endpoints $\{1,\ldots,n\}\setminus \{\gamma_R\}\to \{1,\ldots,n-1\}$.
    We conclude in the same manner.
\end{proof}
\begin{eg}
    Take the NAF $F=\{(1,3),(1,2)\}$ for $S_3$ and $\sigma=231$ in one-line notation. Then
    $$(\partial_{(1,3),(1,2)}f)(t_2,t_3,t_1;t_1,t_2,t_3)=\partial_{(1,2)}(\tope{1,3}f)(t_2,t_1;t_1,t_2)=\partial_{\emptyset}(\tope{1,1}\tope{1,3}f)(t_2;t_2)=\rope{1,2}\tope{1,1}\tope{1,3}f.$$
    In the last two equalities because we had a choice on which endpoint was free we could have also expressed $\partial_{(1,2)}(\tope{1,3}f)(t_2,t_1;t_1,t_2)=\partial_{\emptyset}(\tope{1,2}\tope{1,3}f)(t_1;t_1)=\rope{1,1}\tope{1,2}\tope{1,3}f$.
\end{eg}


    


\begin{cor}\label{th:two_three_term_relations}
    For $a<b$ we have  
    \begin{align*}
    \topeik{i}{b}\ropeik{i+1}{a}&=\ropeik{i+1}{b}\topeik{i}{a}+\ropeik{i}{a}\topeik{i+1}{b}+(t_b-t_a)\topeik{i}{b}\topeik{i}{a}\\\topeik{i}{a}\ropeik{i+1}{b}&=\ropeik{i+1}{b}\topeik{i}{a}+\ropeik{i}{a}\topeik{i+1}{b}+(t_b-t_a)\topeik{i}{a}\topeik{i+1}{b}.
    \end{align*}
\end{cor}
\begin{proof}
    These identities follow from the identities on long-range divided differences \begin{align*}\partial_{i,i+2}=&s_i\partial_{i+1}+\partial_{i}+(x_{i+2}-x_{i+1})s_i\partial_{i+1}\partial_{i}\\\partial_{i,i+2}=&s_{i+1}s_i\partial_{i+1}+s_{i+1}\partial_{i}+(x_{i+1}-x_{i+2})s_{i+1}\partial_{i}\partial_{i+1},\end{align*}
    after applying $\ropeik{i+1}{a}\ropeik{i+2}{b}$ to the first identity and $\ropeik{i+1}{b}\ropeik{i+2}{a}$ to the second identity.
\end{proof}

We now explain how one can interpret our straightening rules diagrammatically. What follows is a decorated version of the calculus developed in \cite{NST_a}. The generators $\tope{i,k}$ and $\rope{i,k}$ are represented by the strand diagrams in Figure~\ref{fig:generators_strands}.

\begin{figure}[H]
    \centering
    \includegraphics[width=0.8\linewidth]{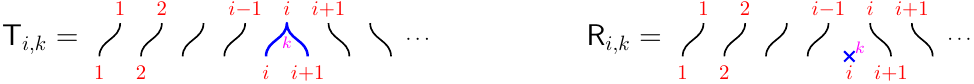}
    \caption{The strand diagrams that represent the generators}
    \label{fig:generators_strands}
\end{figure}
These elementary strand diagrams play the role of building blocks.
Any valid composite may be represented by stacking these strand diagrams one atop the other.
See Figure~\ref{fig:example_composite_strand} for the diagram representing the composite $\rope{1,2}\tope{1,1}\tope{1,3}$ of the preceding example.
\begin{figure}[H]
    \centering
    \includegraphics[width=0.2\linewidth]{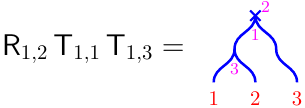}
    \caption{Strand diagram for a valid composite. The numbers in pink correspond to the second index of each operation, read from top to bottom.}
    \label{fig:example_composite_strand}
\end{figure}
\begin{fact}
The commutation relations imply that the composite operation is determined by the combinatorial data of the labeled nested forest structure.
\end{fact}
\begin{proof}
    The only additional data of the stacked diagram is of a linear extension of the poset on $[n]$ determined by the covering relations $a\prec b$ if $a$ is the child of $b$, or $a$ is directly nested under $b$. Two linear extensions of a poset are related by swapping adjacent commuting elements, and the commutation relations enumerate all of these local commutation relations.
\end{proof}

\begin{eg}Figure~\ref{fig:isotopy_strands} shows one decorated nested forest presented as two different valid composites, its three components colour-coded.
The first subscripts change under such an isotopy---the merge decorated $4$ is $\topeik{2}{4}$ on the left and $\topeik{4}{4}$ on the right---while the decorations do not.

\begin{figure}[H]
    \centering
    \includegraphics[width=0.5\linewidth]{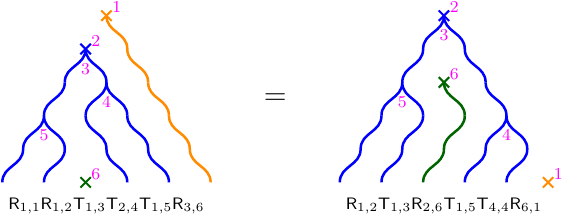}
    \caption{One decorated nested forest presented as two different valid composites}
    \label{fig:isotopy_strands}
\end{figure}
\end{eg}

This graphical perspective recasts the straightening rules in Corollary~\ref{th:two_three_term_relations} as shown in Figure~\ref{fig:magic_strands}.

\begin{figure}[H]
    \centering
    \includegraphics[scale=0.8]{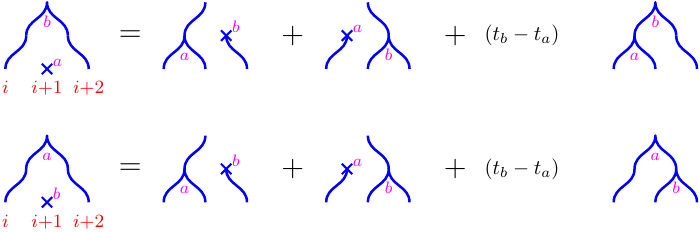}
    \caption{The Graham-positive three term straightening relations}
    \label{fig:magic_strands}
\end{figure}

\begin{eg}Figure~\ref{fig:big_boring_example} straightens the valid composite $\rope{1,6}\tope{1,1}\rope{1,2}\tope{2,5}\tope{2,3}\rope{3,4}$ to operations corresponding to non-nested forests. At each step, the two letters in red are those to which we apply the rules in Figure~\ref{fig:magic_strands}. The terms within the shaded boxes are those for which no further straightening is needed and thus represent final terms in the rewriting procedure.
\begin{figure}[!htbp]
    \includegraphics[scale=0.8]{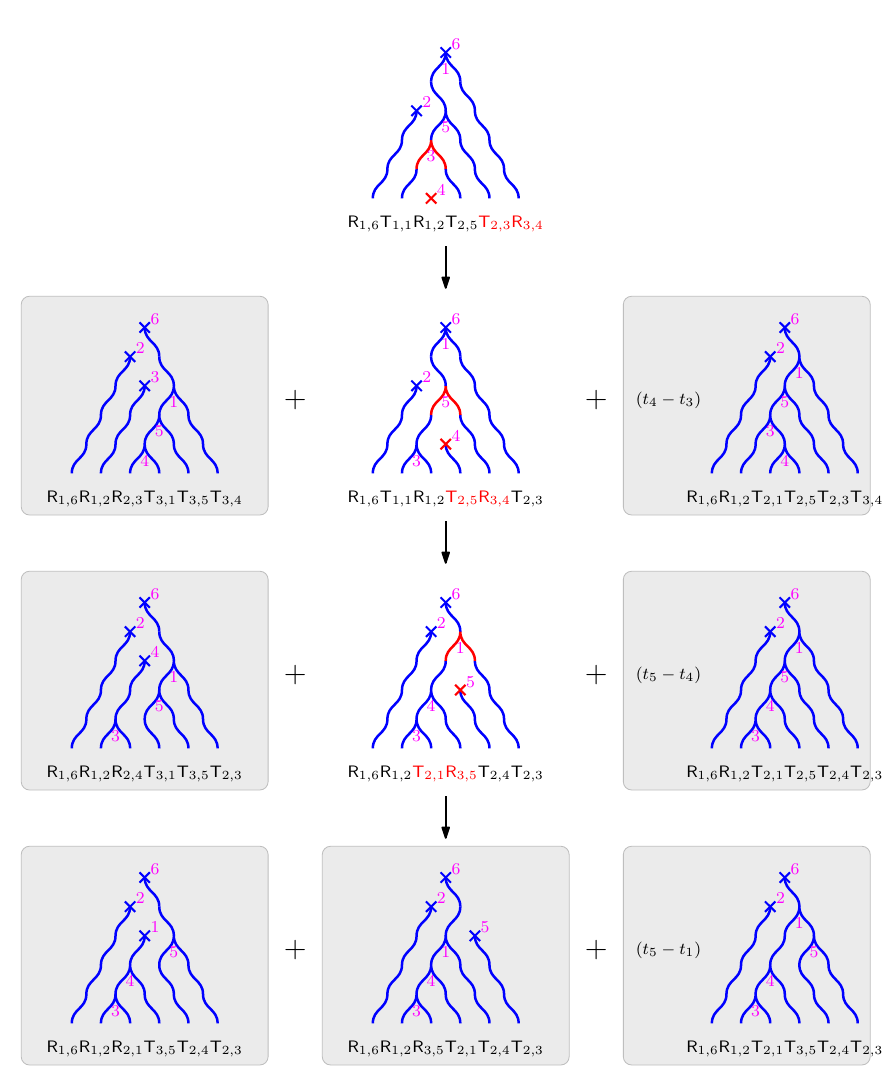}
    \caption{Straightening the valid composite $\rope{1,6}\tope{1,1}\rope{1,2}\tope{2,5}\tope{2,3}\rope{3,4}$ to non-nesting forests.}
    \label{fig:big_boring_example}
\end{figure}
\end{eg}
We can now reinterpret the algorithm from Theorem~\ref{thm:upartialFpositive} to expand $\ev_u\partial_F$ into a Graham-positive combination of $\ev_{\idem}\partial_w$ operations as was needed for Corollary~\ref{cor:partialFevalpos}. By \Cref{thm:validcomposite} we can interpret $\ev_u\partial_F$ as a valid composite of $\rope{i,k}$ and $\tope{i,k}$ operations. As described above, we can straighten the operation into a Graham-positive combination of operations associated to non-nested forests. Any non-empty forest has a node with children $i,i+1$ for some $i$, and applying the commutation relations we can assume $\tope{i,k}$ is the last operation in the valid composite. Replacing $\tope{i,k}$ with either $\rope{i,k}\partial_i$ or $\rope{i+1,k}\partial_i$ we can now iterate the straightening on the prefix before $\partial_i$. Eventually every term will be of the form $f(t_2-t_1,t_3-t_2,\ldots,t_n-t_{n-1})\rope{i_1,k_1}\cdots \rope{i_n,k_n}\partial_w$ where $f$ has nonnegative coefficients. Each $\rope{i_1,k_1}\cdots \rope{i_n,k_n}$ is an operation $\ev_w=\ev_{\idem}w$ for some $w\in S_n$, and we can apply AJS--Billey straightening to $w$ as in the proof of Theorem~\ref{th:ajs-billey} to conclude.

\bibliographystyle{hplain}
\bibliography{sample}
\end{document}